\documentclass[a4paper,12pt]{article}
\usepackage{amsmath,amssymb,enumerate,amsthm,bbm}
\usepackage{color}

\usepackage{hyperref}
\hypersetup{
setpagesize=false,
 bookmarksnumbered=true,%
 bookmarksopen=true,%
 colorlinks=true,%
 linkcolor=blue,
 citecolor=red
}

\newcommand{\R}{\mathbb{R}}

\newcommand{\N}{\mathbb{N}}
\newcommand{\ep}{\varepsilon}
\newcommand{\pa}{\partial}

\newtheorem{theorem}{Theorem}[section]
\newtheorem{lemma}[theorem]{Lemma}

\newtheorem{corollary}[theorem]{Corollary}
\theoremstyle{remark}
\newtheorem{remark}{Remark}[section]
\theoremstyle{definition}

\newtheorem{definition}{Definition}[section]

\numberwithin{equation}{section}

\usepackage{calc}
\makeatletter
\def\@cite#1#2{[{{\bfseries #1}\if@tempswa , #2\fi}]}
\newcounter{hours}\newcounter{minutes}
\renewcommand*{\thehours}{\two@digits\c@hours}
\renewcommand*{\theminutes}{\two@digits\c@minutes}
\makeatother                  

\begin{document}
\begin{center}
\Large{{\bf 
On the decay of mass with respect to 
\\
an invariant measure for semilinear 
\\
heat equations in exterior domains 
\\
}}
\end{center}

\vspace{5pt}

\begin{center}
Ahmad Fino%
\footnote{
College of Engineering and Technology, American University of the Middle East, Egaila 54200, Kuwait, 
E-mail:\ {\tt ahmad.fino@aum.edu.kw}}
and
Motohiro Sobajima%
\footnote{
Department of Mathematics, 
Faculty of Science and Technology, Tokyo University of Science,  
2641 Yamazaki, Noda-shi, Chiba, 278-8510, Japan,  
E-mail:\ {\tt msobajima1984@gmail.com}}

\end{center}

\newenvironment{summary}{\vspace{.5\baselineskip}\begin{list}{}{%
     \setlength{\baselineskip}{0.85\baselineskip}
     \setlength{\topsep}{0pt}
     \setlength{\leftmargin}{12mm}
     \setlength{\rightmargin}{12mm}
     \setlength{\listparindent}{0mm}
     \setlength{\itemindent}{\listparindent}
     \setlength{\parsep}{0pt}
     \item\relax}}{\end{list}\vspace{.5\baselineskip}}
\begin{summary}
{\footnotesize {\bf Abstract.}
The paper concerns with the decay property of 
solutions to the initial-boundary value problem of the semilinear heat equation $\pa_tu-\Delta u+u^p=0$ in exterior domains $\Omega$ in $\R^N$ ($N\geq 2$). 
The problem for the one-dimensional case is formulated with 
$\Omega=(0,\infty)$ which is one of the representative of the connected components in $\R$.
One can see that the $C_0$-semigroup for the corresponding linear problem
possesses an invariant measure $\phi(x)\,dx$, where $\phi$ is a positive harmonic function satisfying the Dirichlet boundary condition. 
This paper clarifies 
that the mass of solutions with respect to the measure $\phi(x)\,dx$ 
vanishes as $t\to \infty$ if and only if $1<p\leq \min\{2,1+\frac{2}{N}\}$. 
In the other case $p>\min\{2,1+\frac{2}{N}\}$, 
we prove that all solutions are asymptotically free. The asymptotic profile 
is actually given by a modification with Gaussian when $N\geq 3$. 
}
\end{summary}




\section{Introduction}
This paper concerns the initial-boundary value problem of 
the semilinear heat equation in an exterior domain
\begin{equation}\label{intro:eq1}
\begin{cases}
\pa_tu(t,x)-\Delta u(t,x)=-\big( u(t,x) \big)^p 
&\text{in}\ (0,\infty)\times \Omega, 
\\
u(t,x)=0
&\text{on}\ (0,\infty)\times\pa \Omega, 
\\
u(0,x)=u_0(x)\geq0 
&\text{in}\ \Omega, 
\end{cases}
\end{equation} 
where the unknown function $u:[0,\infty)\times \Omega\to \R$ is nonnegative, 
$N\geq 1$, $p>1$, and $\Omega\subset\R^N$ 
be a connected exterior domain whose obstacle 
$\mathcal{O}=\Omega^c\subset\R^N$ 
is bounded with smooth compact boundary 
$\partial\Omega$. 
The initial value $u_0$ belongs to 
a class of bounded and integrable functions, 
which is specified later. 
Without loss of generality, we assume that $0$ is located in the interior of $\mathcal{O}$ and so we put
$$R_0:=\sup\{|x|;\,\,x\in \mathcal{O}\}>0,\qquad r_0=\inf\{|x|;\,\,x\in \Omega\}>0.$$
This gives $\overline{B(0,r_0)}\subset\mathcal{O}\subset\overline{B(0,R_0)}$, where $B(0,r)=\{x\in\mathbb{R}^N;\,\,|x|<r\}$. Note that, for the sake of simplicity, 
for the one-dimensional case 
we may consider $\Omega=(0,\infty)$.
It is worth noting that the existence 
of global-in-time nonnegative (classical) solutions 
are obtained by the standard procedure via the comparison principle.

One of important properties
of the heat equation 
\begin{equation}\label{intro:heatRN}
\begin{cases}
\pa_tu(t,x)-\Delta u(t,x)=0
&\text{in}\ (0,\infty)\times \R^N, 
\\
u(0,x)=u_0(x)\geq0 
&\text{in}\ \R^N
\end{cases}
\end{equation} 
is the conservation of the mass 
of solutions defined as
\begin{align}\label{intro:mass}
M(u(t))=\int_{\R^N}u(t,x)\,dx
\end{align}
whenever the solution $u$ (and therefore, $u_0$) is integrable
in $\R^N$. 
This is easily verified by the representation formula
\begin{equation}\label{intro:Gaussian}
u(t,x)=[G(t,\cdot)*u_0](x), \quad 
G(t,x)=(4\pi t)^{-\frac{N}{2}}e^{-\frac{|x|^2}{4t}}. 
\end{equation}
In the case of exterior domains, the situation changes. 
In the recent papers by
Dom\'{i}nguez-de-Tena and Rodr\'{i}guez-Bernal \cite{DR2024,DR2025JDE},
the authors discussed 
the effect of 
several boundary conditions
for the linear heat equation 
in exterior domains. Roughly speaking, 
the mass for the initial-boundary value problem
\begin{equation}\label{intro:eq2:linear}
\begin{cases}
\pa_tu(t,x)-\Delta u(t,x)=0 
&\text{in}\ (0,\infty)\times \Omega, 
\\
u(t,x)=0
&\text{on}\ (0,\infty)\times\pa \Omega, 
\\
u(0,x)=u_0(x)\geq0 
&\text{in}\ \Omega
\end{cases}
\end{equation}
 is no longer preserved:
\[
\frac{d}{dt}\int_{\Omega}u(t,x)\,dx=
\int_{\Omega} \Delta u(t,x)\,dx
=
\int_{\pa \Omega} \nabla u(t,x)\cdot\boldsymbol{\nu}(x)\,dS(x)<0,
\]
where $\boldsymbol{\nu}(x)$ denotes the outer normal vector at $x\in\pa \Omega$. 
Instead of this, one can find the following asymptotic profile of the mass
\[
\lim_{t\to \infty}\int_{\Omega}u(t,x)\,dx
=
\int_{\Omega}u_0(x)\phi(x)\,dx, \quad \phi(x)=\lim_{t\to \infty}e^{t\Delta_{\Omega}}\mathbbm{1}_{\Omega}(x),
\]
where $\mathbbm{1}_K$ is the indicator function on $K$ 
and 
$e^{t\Delta_{\Omega}}$ is the Dirichlet heat semigroup on $\Omega$.  
Namely, $e^{t\Delta_{\Omega}}u_0$ denotes the unique solution of \eqref{intro:eq2:linear}.
It should be emphasized that the function $\phi$ can be characterized as 
a (bounded) harmonic function in $\Omega$ satisfying the Dirichlet boundary condition,
which is non-trivial if $N\geq 3$. 
Therefore in \cite{DR2024,DR2025JDE}, the quantity 
\[
M_{\Omega}(u(t))=\int_{\Omega}u(t,x)\phi(x)\,dx
\]
(which was named ``asymptotic mass'')
plays a central role to clarify the asymptotic profile of solutions to \eqref{intro:eq2:linear}.
By using $M_{\Omega}(u(\cdot))$, they proved that if $N\geq 3$, 
then the solution $u$ of the linear problem \eqref{intro:eq2:linear} 
satisfies 	 
\[
\lim_{t\to \infty}
\Big(t^\frac{N}{2}\|u(t,\cdot)-M_{\Omega}(u_0)\phi(\cdot)G(t,\cdot)\|_{L^\infty(\Omega)}\Big)=0.
\]
Note that, in \cite{DR2025JDE},
the asymptotic behavior of solutions for $N=1,2$ was not addressed
since their auxiliary function $\phi$ is identically zero. 
This phenomenon suggests that 
the asymptotic profile of solutions to exterior problem \eqref{intro:eq2:linear} 
is completely different from that of the problem in $\R^N$ when $N=1,2$. 

In the analysis of \eqref{intro:eq1}, 
the interest of the present paper 
is the effect of the semilinear term $-u^p$
for the structure of the corresponding linear 
problem \eqref{intro:eq2:linear}.
In Gmira--V\'{e}ron \cite{GmiraVeron1984}, 
the problem with $\Omega=\R^N$ is studied 
and it is observed that 
comparing two supersolutions 
$(t,x)\mapsto [e^{t\Delta}u_0](x)$ 
(linear effect)
and 
$t\mapsto (p-1)^{-\frac{1}{p-1}}t^{-\frac{1}{p-1}}$ (nonlinear effect), 
one can see that the situation 
differs whether $p$ is 
larger or smaller than the 
threshold $1+\frac{2}{N}$ for the integrable initial data. 
The situation 
in the sense of the effect from the nonlinearlity
seems quite close to 
the one for the existence/non-existence of
global-in-time non-trivial solutions to 
the problem $\pa_t u-\Delta u=u^p$ 
developed by Fujita \cite{Fujita1966} 
and subsequent papers 
(see e.g., Hayakawa \cite{Hayakawa1973}, 
Sugitani \cite{Sugitani1975}, 
Kobayshi--Sirao--Tanaka \cite{KobayashiSiraoTanaka1977}, 
and also a monograph 
by Souplet--Quitnner \cite{QS}).
In this case, 
the large time behavior of the mass 
defined by \eqref{intro:mass}
plays a crucial role for understanding the asymptotic behavior of solution $u$;
\begin{itemize}
\item if $1<p\leq 1+\frac{2}{N}$, 
then the mass vanishes at $t\to \infty$, 
that is,  
$\lim\limits_{t\to \infty}M(u(t))=0$;
\item if $p>1+\frac{2}{N}$, 
then the mass does not vanish at $t\to \infty$
and the solution $u$ be haves like a solution
of linear heat equation. Namely, 
$M_\infty=\lim\limits_{t\to \infty}M(u(t))>0$
and $\lim\limits_{t\to\infty}
|t^{\frac{N}{2}}u(t,x)-M_{\infty}G(t,x)|=0$
uniformly on the sets 
$E_c=\{x\in\R^N\;;\;|x|\leq ct^{1/2}\}$, 
where $c$ is an arbitrary positive constant.
\end{itemize}
Note that the following estimate 
was also found in Fino--Karch \cite{FinoKarch2010}: 
\[
\lim_{t\to\infty}\Big(t^{\frac{N}{2}(1-\frac{1}{q})}\|u(t,\cdot)-M_{\infty}G(t,\cdot)\|_{L^q(\R^N)}\Big)=0
\]
for every $1\leq q<\infty$. 
On the other hand, 
non-existence of global-in-time solutions 
to the two-dimensional semilinear problem 
\begin{equation}\label{intro:eq3:twodim-fujita}
\begin{cases}
\pa_tu(t,x)-\Delta u(t,x)=\big(u(t,x)\big)^2 
&\text{in}\ (0,\infty)\times \big(\R^2\setminus \overline{B(0,1)}\big), 
\\
u(t,x)=0
&\text{on}\ (0,\infty)\times\pa B(0,1), 
\\
u(0,x)=u_0(x)\geq0 
&\text{in}\ \R^2\setminus \overline{B(0,1)},
\end{cases}
\end{equation}
is shown in Ikeda--Sobajima \cite{IkedaSobajima2019} by the test function method 
via the use of the quantity $\int_{\Omega}u(t,x)\log |x|\,dx$.
Actually, Grigor'yan and Saloff-Coste 
pointed out in \cite{GS2002} that 
the heat kernel $k_{\Omega}$ of the linear problem 
\eqref{intro:eq2:linear} with
$\Omega=\R^2\setminus \overline{B(0,1)}$
satisfies 
\[
k_{\Omega}(t,x,y)\sim 
\frac{(1+\log|x|)(1+\log|y|)}
{(1+\log|x|+\log \sqrt{t})
(1+\log|x|+\log \sqrt{t})}
\times
G(t,x), \ \ t\geq 1, |x|,|y|\gg 1. 
\]
This provides that the usual mass 
for $e^{t\Delta_\Omega}u_0$
(the solution of the linear problem \eqref{intro:eq2:linear})
has a logarithmic decay:
\[
\int_{\Omega}e^{t\Delta_{\Omega}}u_0(x)\,dx\leq \frac{C}{1+\log (1+t)}\int_{\Omega}u_0(x)(1+\log |x|)\,dx, \quad t\geq 0
\]
(an alternative proof can be found in 
Ikeda--Sobajima--Taniguchi--Wakasugi \cite{IkedaSobajimaTaniguchiWakasugi2024} 
by the comparison principle). 
In the background of the above phenomenon, 
the recurrence property of the Brownian motion in $\R^2$ 
plays an important role.
Roughly speaking, no Brownian motions escape to the spatial infinity
(all Brownian motions trapped to the obstacle $\mathcal{O}$);
note that this aspect can be seen in Pinsky \cite{Pinsky2000}. 
For the study of one-dimensional case, so-called first momentum 
$\int_0^\infty u(t,x) x \,dx$ is effectively used in 
the analysis of a non-local problem 
in Cort\'{a}zar--Elgueta--Quir\'{o}s--Wolanski \cite{CEQW2016}.
A further detailed analysis for heat kernels in two-dimensional case
via the logarithmic Sobolev inequality can be found in \cite{CGQ_arxiv}.

Here we present a different perspective. 
It is remarkable that 
the weighted measure $d\mu=\phi(x)\,dx$ 
with the following choice 
actually acts as an invariant measure for the semigroup $e^{t\Delta_{\Omega}}$:
\begin{equation}\label{intro:phi_equation}
\begin{cases}
\Delta \phi=0 &\text{in}\ \Omega,  
\\
\phi=0 &\text{on}\ \pa\Omega,  
\\
\phi(x)\sim 
\phi_0(x)=
\begin{cases}
x&\text{if}\ N=1, 
\\
\log |x| &\text{if}\ N=2, 
\\
1-|x|^{2-N} &\text{if}\ N\geq 3 
\end{cases}
&\text{as}\ t\to \infty
\end{cases}
\end{equation}
and $\phi_0$ is the explicit harmonic function vanishing on $\pa B(0,1)$ $(N\geq 2)$ or 
at $x=0$ $(N=1)$. 
Namely, 
the mass of $e^{t\Delta_{\Omega}}u_0$ with respect to the measure $d\mu$ 
is preserved: 
\[
\int_{\Omega}e^{t\Delta_\Omega}u_0\,d\mu=\int_{\Omega}u_0\,d\mu, \quad t\geq 0. 
\]
From this viewpoint,
we can analyze the asymptotic behavior for our target problem \eqref{intro:eq1}
in a unified manner for both cases $N\geq 3$ and $N=1,2$.
Then one can expect that $M_\Omega(u(\cdot))$ has 
crucial information for the large time behavior for \eqref{intro:eq1} 
for all dimensions.

In the present paper, we shall discuss the effect 
from the nonlinearity in the problem \eqref{intro:eq1} via the invariant measure $d\mu=\phi(x)\,dx$. 
More precisely, we focus our attention to the limit 
of the mass $M_{\Omega}(u(t))$ with respect to the invariant measure. 
Up to our knowledge, the cases $N=1,2$ have not been studied before. We will show, in the case $N=2$, a logarithmic decay of solution as expected due of the one mentioned by Grigor’yan and Saloff-Coste in \cite{GS2002} while for the case $N=1$ polynomial decay will be shown by using 
the translation into three-dimensional radially symmetric solutions of \eqref{intro:heatRN}.

Before stating the main result in the present paper, 
we give the definition of the solution of \eqref{intro:eq1}.

\begin{definition}[Mild solution]
Let $u_0\in L^1(\Omega)\cap L^\infty(\Omega)$ and $T>0$. 
A function $u:[0,T)\times \Omega\to \R$ is said to be a {\it mild solution} 
of \eqref{intro:eq1} in $(0,T)$ 
if $u$ satisfies the integral equation
\begin{equation}\label{integralform}
    u(t,x)=e^{t\Delta_{\Omega}}u_0(x)
    -\int_0^t\big(e^{(t-s)\Delta_{\Omega}}u(s,\cdot)^p\big)(x)\,ds
\end{equation}
for all $x\in \Omega$ and $t\in(0,T)$. 
If $u:[0,\infty)\times \Omega\to \R$
is a mild solution of \eqref{intro:eq1} 
in $(0,T)$ for all $T>0$, then 
$u$ is called global-in-time (mild) solution
of \eqref{intro:eq1}.
\end{definition}
\begin{remark}
The standard bootstrap argument gives that 
a mild solution $u$ of \eqref{intro:eq1} 
in $[0,T]$ have $C^1$-regularity with respect to $t$ and $C^2$-regularity with respect to $x$.
Therefore as a consequence, $u$ satisfies $\pa_tu-\Delta u=-u^p$ in the pointwise sense.
Since the nonlinearity is of absorbing type, 
by comparison principle we always have 
a global-in-time solution which satisfies 
$0\leq u(t,x)\leq \|u_0\|_{L^\infty(\Omega)}$ whenever $u_0$ is nonnegative.
\end{remark}

Now we are in a position to state the main result of the present paper
about the large time behavior of 
\[
M_{\Omega}(u(t))
=\int_{\Omega}u(t,x)\phi(x)\,dx
=\int_{\Omega}u_0(x)\phi(x)\,dx
-
\int_0^t\int_\Omega\big(u(s,x)\big)^p\phi(x)\,dx\,ds
\]
which is nonnegative and nonincreasing, and therefore, 
the limit as $t\to\infty$ always exists.

\begin{theorem}\label{thm:main}
Let $\phi$ satisfy \eqref{intro:phi_equation}.
Assume $u_0\in L^\infty(\Omega)$ satisfies 
$u_0\geq 0$, $u_0\not\equiv0$ and $M_{\Omega}(u_0)<\infty$. 
Then the corresponding global-in-time solution $u$ of \eqref{intro:eq1}
satisfies
\[
\lim_{t\to\infty}M_{\Omega}(u(t))
\begin{cases}
=0
&\text{if}\ 1<p\leq \min\{2,1+\frac{2}{N}\}, 
\\
>0
&\text{if}\ p>\min\{2,1+\frac{2}{N}\}.
\end{cases}
\]
\end{theorem}

According to Theorem \ref{thm:main}, 
we can assert that 
in the case $p\leq \min\{2,1+\frac{2}{N}\}$
the effect of nonlinearity is dominant. 
In the other case $p>\min\{2,1+\frac{2}{N}\}$, 
the problem \eqref{intro:eq1} still seems to have some properties of 
the corresponding linear problem \eqref{intro:eq2:linear}.
The following theorem 
clarifies the linear asymptotic profile of solutions of \eqref{intro:eq1}.

\begin{theorem}[Asymptotic behavior]
Let $u$ be as in Theorem \ref{thm:main}. 
If $p>\min\{2,1+\frac{2}{N}\}$, then
$u$ has the following asymptotic behavior:
for every $1\leq q\leq \infty$
\[
\|u(t,\cdot)-e^{t\Delta_\Omega}u_\infty\|_{L^q(\Omega)}=
\begin{cases}
o\big(t^{-(1-\frac{1}{2q})}\big)
&\text{if}\ N=1, 
\\
o\big( t^{-(1-\frac{1}{q})}(\log t)^{-1}\big)
&\text{if}\ N=2, 
\\
o\big(t^{-\frac{N}{2}(1-\frac{1}{q})}\big)
&\text{if}\ N\geq 3 
\end{cases}
\]
as $t\to \infty$, where $u_\infty$ is given by 
\[
u_{\infty}(x)=u_0(x)-\int_{0}^\infty \big(u(t,x)\big)^p\,dt, \quad x\in \Omega.
\]
\end{theorem}
\begin{remark}
Although $u_\infty$ may changes its sign, in this case we can see by the monotone convergence theorem that 
$u_\infty$ also satisfies $M_{\Omega}(|u_\infty|)<+\infty$ with 
\[
M_{\Omega}(u_\infty)=\int_{\Omega}u_\infty(x)\phi(x)\,dx=\lim\limits_{t\to\infty}M_{\Omega}(u(t))>0. 
\]
By adopting the argument in \cite[Section 3]{Sobajima_preprint}, 
the following profile can be shown:
\[
\begin{cases}
\limsup\limits_{t\to\infty}
\Big(t\|e^{t\Delta_{\Omega}}u_\infty\|_{L^\infty(\Omega)} \Big)>0
&\text{if}\ N=1, 
\\
\limsup\limits_{t\to\infty}
\Big(t(\log t)\|e^{t\Delta_{\Omega}}u_\infty\|_{L^\infty(\Omega)} \Big)>0
&\text{if}\ N=2, 
\\
\limsup\limits_{t\to\infty}
\Big(t^{\frac{N}{2}}\|e^{t\Delta_{\Omega}}u_\infty\|_{L^\infty(\Omega)} \Big)>0
&\text{if}\ N\geq 3.
\end{cases}
\]
Therefore 
(although the description of asymptotic profiles is not unique),
the profile $e^{t\Delta_\Omega}u_\infty$ can be justified as the asymptotic profile of $u$.
\end{remark}

Once we find the linear asymptotic behavior, 
it turns out that the result 
for the precise asymptotics in \cite{DR2025JDE} is applicable.
As a consequence, we find the following assertion. 
\begin{corollary}
\label{cor:Ngeq3}
If $N\geq 3$ and $p>1+\frac{2}{N}$, then 
the solution $u$ in Theorem \ref{thm;asymptotics} further satisfies 
that for every $1\leq q\leq \infty$, 
\[
\|u(t,\cdot)-M_{\Omega}(u_\infty)\phi(\cdot) G(t,\cdot)\|_{L^q(\Omega)}=o\big(t^{-\frac{N}{2}(1-\frac{1}{q})}\big)
\]
as $t\to\infty$, where $M_{\Omega}(u_\infty)=\lim\limits_{t\to\infty}M_{\Omega}(u(t))>0$. 
\end{corollary}

Now we briefly explain the idea for studying the large time behavior of $M_{\Omega}(u(\cdot))$.
In the study of the Fujita equation $\pa_tu-\Delta u=u^p$, 
it is effective to use supersolutions of the form $h(t)e^{t\Delta_\Omega}u_0$
to construct a global solution in the super-critical case. 
On the other hand, to show the blowup result for the Fujita equation
in the sub-critical and critical cases, 
so-called test function method is useful. 
Roughly speaking, the idea of the present paper 
is to adopt the above argument 
to the study of \eqref{intro:eq1} which has the opposite sign in front of the nonlinearity. 
Actually, functions of the form $h(t)e^{t\Delta_\Omega}u_0$ 
are helpful to find a sub-solution of \eqref{intro:eq1}, 
which has a positive mass with respect to the invariant measure $\phi(x)\,dx$
when 
$p$ is super-critical. 
If  $p$ is sub-critical or critical, then 
the vanishing $M_{\Omega}(u(t))\to 0$ $(t\to\infty)$ 
is proved by a modification of the test function method. 
The asymptotic behavior for the super-critical case
is verified via an argument similar to that in \cite{Sobajima2025JMAA}.

The organization of the present paper is as follows: 
In Section \ref{sec:prelim},
we summarize the fundamental properties of 
positive harmonic functions in $\Omega$ 
satisfying the Dirichlet boundary condition
and the decay estimates for the corresponding Dirichlet heat semigroup $e^{t\Delta_\Omega}$. 
In 
Sections \ref{sec:non-vanishing}
and 
\ref{sec:vanishing}, 
we discuss the case of non-vanishing 
and vanishing of the mass with respect to the invariant measure
via the argument explained above, respectively. 
Finally, we prove 
the linear asymptotic profile of solutions to \eqref{intro:eq1} 
in Section \ref{sec:asymptotics}.

\section{Preliminaries}\label{sec:prelim}

\subsection{Harmonic functions and 
Dirichlet heat semigroups}
In this subsection, we give a reasonable
harmonic function $\phi$ satisfying 
the same behavior as $\phi_0$, 
which will be used in the proof of the main theorem
(for the reader's convenience we refer 
 \cite[Lemma 2.2]{Han2} for $N=1$, 
\cite[Lemma 3.2]{Sobajima2021} for $N=2$
and 
\cite[Lemma 2.11]{Fino1} for $N\geq 3$).

\begin{lemma}\label{lem:harmonic}
The following assertions hold:
\begin{itemize}
\item[\bf (i)] 
If $\Omega=(0,\infty)$, then 
$\phi(x)=\phi_0(x)=x$ is the unique positive harmonic function 
satisfying $\phi(0)=0$ and 
$\lim\limits_{x\to \infty}\frac{\phi(x)}{\phi_0(x)}=1$.
\item[\bf (ii)] 
If $\Omega$ is an exterior domain in $\R^2$, then there exists a unique positive harmonic function $\phi$ satisfying $\lim\limits_{|x|\to \infty}\frac{\phi(x)}{\phi_0(x)}=1$. 
Moreover, there exists a positive constant $R_{\phi}>1$ such that 
for every $x\in\Omega$ with $|x|\geq R_{\phi}$,  
\begin{gather*}
 \frac{1}{2}\phi_0(x)\leq \phi(x)\leq 2\phi_0(x), \quad\frac{1}{2}\leq x\cdot \nabla \phi(x)\leq 2.
 \end{gather*}
\item[\bf (iii)] 
If $\Omega$ is an exterior domain in $\R^N$ 
$(N\geq 3)$, then there exists a unique positive harmonic function $\phi$ satisfying $\lim\limits_{|x|\to \infty}\frac{\phi(x)}{\phi_0(x)}=1$. 
Moreover, 
$\phi$ satisfies $0<\phi<1$ on $\Omega$ 
and 
there exists a positive constant $R_{\phi}>1$ such that 
for every $x\in\Omega$ with $|x|\geq R_{\phi}$,  
\[
|\nabla \phi(x)|\leq \frac{C}{|x|^{N-1}}. 
\]
\end{itemize}
In the cases {\bf (ii}) and {\bf (iii)}, one has for $x\in \Omega$ with $|x|>R_{0}$,
\[
\phi_0\left(\frac{x}{R_0}\right)\leq \phi(x)\leq \phi_0\left(\frac{x}{r_0}\right).
\]
\end{lemma}

Consider the homogeneous initial boundary value problem of the heat equation in an exterior domain
 \begin{equation}\label{eq:linearheat}
\begin{cases}
\pa_t u-\Delta u =0 & \text{in}\ (0,\infty)\times \Omega,
\\
u=0 & \text{on}\ (0,\infty)\times \pa\Omega, 
\\
u(0,x)=  u_0(x)
&\text{on}\ \Omega
\end{cases}
\end{equation} 
with $u_0\in L^1(\Omega)\cap L^\infty(\Omega)$. 
The existence and uniqueness of solutions 
are verified via the Hille--Yosida theorem. 
Here we use the notation 
\[
u(t)=e^{t\Delta_\Omega}u_0=S(t)u_0.
\]
The positivity preserving property of $S(t)$ is a consequence of the comparison principle.

The following lemma describes that 
the harmonic functions constructed in \cite{DR2025JDE}
are actually the same as our framework when $N\geq 3$. 
\begin{lemma}\label{lem:unique-harmonic}
If $N\geq 3$, then $\lim\limits_{t\to \infty}\Big(S(t)\mathbbm{1}_{\Omega}(x)\Big)=\phi(x)$ for all $x\in \Omega$.
\end{lemma}
\begin{proof}
Put $\phi_*(x)=\lim\limits_{t\to \infty}S(t)\mathbbm{1}_{\Omega}(x)$ $(x\in\Omega)$
for simplicity. 
First, note that $S(t)\phi_*=\phi_*$. 
Therefore since $\pa_t\phi_*=0$, it follows that 
$\phi_*$ is harmonic on $\Omega$. 
On the one hand, 
since $S(t)\phi=\phi$
the comparison principle with $\phi\leq 1$ gives
\[
\phi(x)=S(t)\phi(x)\leq S(t)\mathbbm{1}_{\Omega}(x)
\]
and therefore letting $t\to \infty$, we have $\phi(x)\leq \phi_*(x)$ for all $x\in\Omega$.
For the opposite side inequality, we use $\lim\limits_{|x|\to \infty}\phi(x)=1$.
For every $\ep>0$, 
there exists $R_\ep>R_0$ such that 
\[
1-\ep \leq \phi(x) \leq \phi_*(x)\leq 1\quad \text{on}\ \R^N\setminus B(0,R_\ep).
\]
This shows that 
\[
\begin{cases}
\phi - \phi_*=0& \text{on}\ \pa\Omega
\\
\phi - \phi_* \geq 1-\ep - \phi_*\geq -\ep & \text{in}\ \R^N\setminus B(0,R_\ep).
\end{cases}
\]
The maximum principle (in $\Omega\cap B(0,R_\ep)$) provides 
\[
\phi -\phi_*\geq -\ep \quad\text{in}\ \Omega=(\Omega\cap B(0,R_\ep))\cup 
(\R^N\setminus B(0,R_\ep)).
\]
Since $\ep$ is arbitrary, we can find  $\phi -\phi_*\geq 0$. The proof is complete.
\end{proof}

The following lemma is the standard $L^p$-$L^q$ estimates for the Dirichlet 
heat semigroup with the comparison 
$S(t)=e^{t\Delta_\Omega}\leq e^{t\Delta_{\R^N}}$.

\begin{lemma}\label{lem:LpLq-usual}
Let $N\geq 1$. 
If $1\leq p\leq q\leq \infty$, 
there exists a positive constant $C$ 
(depending only on $N,p,q$) such that
for every $f\in L^p(\Omega)$, 
\begin{align}
\label{eq:LpLq-usual}
\|S(t)f\|_{L^q(\Omega)}
&\leq C t^{-\frac{N}{2}(\frac{1}{p}-\frac{1}{q})}\|f\|_{L^p(\Omega)},\quad t>0,
\\
\label{eq:LpLq-usual_Lap}
\|\Delta S(t)f\|_{L^q(\Omega)}
&\leq C t^{-\frac{N}{2}(\frac{1}{p}-\frac{1}{q})-1}\|f\|_{L^p(\Omega)},\quad t>0.
\end{align}
\end{lemma}
In order to describe an exceptional behavior of 
the (mainly two-dimensional) Dirichlet heat semigroups, 
we use the weighted $L^q$-spaces
\[
L_{\phi}^q(\Omega):=
L^q(\Omega,\phi(x)\,dx)=
\big\{v\in L^q_{\rm loc}(\Omega)\,;\,\,v\phi^{\frac{1}{q}}\in L^q(\Omega)\big\}, \quad q\in[1,\infty)
\]
endowed by the norm $\|v\|_{L^q_{\phi}(\Omega)}=\|v \phi^{\frac{1}{q}}\|_{L^q(\Omega)}$, 
where $\phi$ is given by Lemma \ref{lem:harmonic}.
Note that if $N\geq 3$, then we have $L^\infty(\Omega)\cap L_{\phi}^1(\Omega)=L^\infty(\Omega)\cap L^1(\Omega)$.

The following lemma for the two-dimensional case
is essentially given in \cite{GS2002}. 
For the alternative proof based on comparison principle, see \cite[Appendix]{IkedaSobajimaTaniguchiWakasugi2024}. 
\begin{lemma}\label{lem:LpLq_N=2}
Let $\Omega$ be an exterior domain in $\R^2$. If $1\leq p <\infty$ and $p\leq q\leq \infty$, 
there exists a positive constant $C$ 
(depending  only on $p$ and $q$) such that for every $f\in L^q(\Omega)\cap L_\phi^p(\Omega)$, 
\begin{equation}\label{eq:LpLq_N=2}
\|S(t)f\|_{L^q(\Omega)}\leq C (1+t)^{-(\frac{1}{p}-\frac{1}{q})}\big(1+\log(1+t)\big)^{-\frac{1}{p}}\left(\|f\|_{L^q(\Omega)}+\|f\|_{L^p_{\phi}(\Omega)}\right),\quad t\geq 0.
\end{equation}
\end{lemma}
\begin{remark}
Although 
the special case is only treated in  \cite{IkedaSobajimaTaniguchiWakasugi2024},
it is not difficult to generalize to the case of general exterior domains.
Indeed, since $\Omega \subset \Omega'=\R^2\setminus\overline{B(0,r_0)}$, 
the comparison principle shows that 
$\log(\frac{|x|}{R_0})\leq \phi(x)$ and 
$e^{t\Delta_\Omega}\leq e^{t\Delta_{\Omega'}}|_{\Omega}$ (the restriction to $\Omega$). 
Moreover, the $L^p$-$L^q$ type estimates in \cite{IkedaSobajimaTaniguchiWakasugi2024} 
is applicable to $\Omega'$. 
Combining the above facts, we can verify Lemma \ref{lem:LpLq_N=2} for arbitrary general exterior domains.
\end{remark}

The $L^p$-$L^q$ estimates for 
the one-dimensional case $\Omega=(0,\infty)$ 
are the following. 
\begin{lemma}\label{lem:LpLq_N=1}
Let $\Omega=(0,\infty)$. 
There exists a positive constant $C$ such that 
for every $1\leq q\leq \infty$ and $f\in L_{\phi}^1(\Omega)$, 
\begin{align}
\label{eq:LpLq_N=1_1-q}
\|S(t) f\|_{L^q(\Omega)}
&\leq C\,t^{\frac{1}{2q}-1}\|f\|_{L^1_{\phi}(\Omega)},\quad t>0.
\end{align}
\end{lemma}

\begin{remark}\label{rem:N=1toN=3radial}
We can prove \eqref{eq:LpLq_N=1_1-q} with $q=2$ via 
the following translation. 
Define 
$v:(0,\infty)\times\R^3\to \R$ 
as $v(t,x)=\frac{u(t,r)}{r}$ $(r=|x|)$, where $u(t)=S(t)f$.
Then $v$ is a radially symmetric solution of 
the heat equation in $\R^3$ with $v(0,x)=\frac{f(|x|)}{|x|}$. 
Noting that 
\begin{align*}
\|u(t,\cdot)\|_{L^2(0,\infty)}
&=
\left(\int_0^\infty|u(t,r)|^2\,dr\right)^{\frac{1}{2}}
=
\left(\frac{1}{4\pi}\int_{\R^3}
|v(t,x)|^2\,dx\right)^{\frac{1}{2}}
=
\frac{1}{\sqrt{4\pi}}\|v(t,\cdot)\|_{L^2(\R^3)}.
\end{align*}
We see by the $L^p$-$L^q$ estimate 
in $\R^3$ that 
\begin{equation}\label{eq:1d_wL1-Linfty}
\|u(t,\cdot)\|_{L^2(0,\infty)}\leq 
\frac{C}{\sqrt{4\pi}}
t^{-\frac{3}{4}}\|v(0,\cdot)\|_{L^1(\R^3)}
=
C\sqrt{4\pi}
t^{-\frac{3}{4}}\|f\|_{L_{\phi}^1(0,\infty)}.
\end{equation}
The estimate \eqref{eq:LpLq_N=1_1-q} with $q=\infty$ 
is just a combination of \eqref{eq:1d_wL1-Linfty} 
and the usual one-dimensional $L^2$-$L^\infty$ estimate.
Then 
using \eqref{eq:LpLq_N=1_1-q} with $q=\infty$ 
and $L_{\phi}^1$-contractive property, 
we can also find the estimate \eqref{eq:LpLq_N=1_1-q} with $q=1$ as follows: 
\begin{align*}
\|u(t)\|_{L^1(0,\infty)}
&= \int_{0}^{t^{1/2}}|u(t,x)|\,dx
+ \int_{t^{1/2}}^\infty|u(t,x)|\,dx
\\
&\leq 
 \|u(t)\|_{L^\infty(0,t^{1/2})}\int_0^{t^{1/2}}\,dt
  + t^{-1/2}\int_{\sqrt{t}}^\infty x|u(t,x)|\,dx
\leq 
 C\|f\|_{L_\phi^1(0,\infty)}t^{-1/2}.
 \end{align*}
By interpolation, we have the desired estimates for all $q$.
\end{remark}

Summarizing the above decay estimates, we state 
the fundamental estimate which we will frequently use.
\begin{lemma}\label{lem:linear-Linfty-decay}
There exists a positive constant $C$ such that 
for every $1\leq q\leq \infty$
and 
$f\in L^q(\Omega)\cap L_\phi^1(\Omega)$, 
\[
\|S(t)f\|_{L^q(\Omega)}\leq 
\frac{C \Big(\|f\|_{L^q}+\|f\|_{L_{\phi}^1}\Big)}
{(1+t)^{\frac{N}{2}(1-\frac{1}{q})}\mathcal{E}_{N}(t)},\quad t\geq 0
\]
where 
\[
\mathcal{E}_{N}(t)=
\begin{cases}
(1+t)^{\frac{1}{2}}&\text{if}\ N=1, 
\\
1+\log (1+t)&\text{if}\ N=2, 
\\
1 &\text{if}\ N\geq 3. 
\end{cases}
\]
\end{lemma}
\begin{remark}
To clarify the difference between 
the result in the exterior domain 
and that in the whole space $\R^N$, 
we have introduced the notation $\mathcal{E}_{N}(t)$ 
which describes an additional decay. 
\end{remark}

\subsection{Some basic inequalities}

The following two lemmas 
are the summary of elementary 
computations which will be used
in the discussion. 
\begin{lemma}\label{lem:integral_0-t}
Let $r,m\in\R$. 
For every $t>0$, one has 
\[
\int_0^t(1+s)^{r}\big(1+\log(1+s)\big)^{m}\,ds
\begin{cases}
\leq 
C_{m,r}(1+t)^{r+1}\big(1+\log(1+t)\big)^{m}
& \text{if}\ r>-1, m\in\R, 
\\[5pt]
\leq C_{m,r}\big(1+\log(1+t)\big)^{m+1}&\text{if}\ r=-1,\,m>-1,
\\[5pt]
= \log\big(1+\log(1+t)\big)
&\text{if}\ r=-1, m=-1,
\\[5pt]
\leq C_{m,r}
&\text{otherwise},
\end{cases}
\]
where $C_{m,r}$ is a positive constant depending only on $m$ and $r$.
\end{lemma}

\begin{lemma}\label{lem:integral_t-infty}
Let $r,m\in\mathbb{R}$. For every $t>0$, 
one has 
\[
\int_t^\infty(1+s)^{r}(1+\log(1+s))^{m}\,ds
\begin{cases}
\leq C_{m,r}(1+t)^{r+1}(1+\log(1+t))^{m}
&\text{if}\ r<-1, m\in\R,
\\[5pt]
=C_{m,r}\big(1+\log(1+t)\big)^{m+1}
&\text{if}\ r=-1, m<-1,
\\[5pt]
=\infty
&\text{otherwise},
\end{cases}
\]
where $C_{m,r}$ is a positive constant depending only on $m$ and $r$.
\end{lemma}

\section{The non-vanishing case of $M_{\Omega}(u(t))$}
\label{sec:non-vanishing}

To begin with, 
we state 
the fundamental property of the quantity 
\[
M_{\Omega}(u(t))=\int_{\Omega}u(t,x)\phi(x)\,dx
\]
for the solution $u$ of \eqref{intro:eq1}.
Since the finiteness of the above quantity 
is not trivial when $\phi$ is not bounded, we 
would give a short proof based on the monotone convergence theorem. 

\begin{lemma}\label{lem:M_Omega}
If $u_0\in L^\infty(\Omega)\cap L_{\phi}^1(\Omega)$ is nonnegative, then 
the corresponding global-in-time solution $u$ of \eqref{intro:eq1} satisfies
for every $t>0$, 
\begin{equation}\label{eq:M_Omega}
\int_{\Omega}u(t,x) \phi(x)\,dx
+
\int_0^t
\int_{\Omega}u(s,x)^p \phi(x)\,dx
\,ds
=
\int_{\Omega}u_0(x)\phi(x)\,dx.
\end{equation}

\end{lemma}
\begin{proof}
If $N\geq 3$, then 
it follows from the fact $0\leq \phi\leq1$ that 
$u \phi \in C^1([0,\infty);L^1(\Omega))$ and 
\begin{align*}
\frac{d}{dt}\int_\Omega u(t,x)\phi(x)\,dx
=
\int_\Omega (\Delta u(t,x)-u(t,x)^p)\phi(x)\,dx
=
-\int_\Omega u(t,x)^p\phi(x)\,dx.
\end{align*}
If $N=1$, then as in Remark \ref{rem:N=1toN=3radial} 
we can see that the radially symmetric solution 
$v(t,x)=r^{-1}u(t,r)$ $(r=|x|, x\in \R^3)$ 
of $\pa_tv-\Delta v= u(t,|x|)^{p-1}v$ satisfies $v\in C^1([0,\infty);L^1(\R^3))$ 
and 
\begin{align*}
\frac{d}{dt}\int_0^\infty u(t,r)r\,dr
&=
\frac{1}{4\pi}\frac{d}{dt}\int_{\R^3} v(t,x)\,dx
\\
&=\frac{1}{4\pi}\int_{\R^3}(\Delta v(t,x)-u(t,|x|)^{p-1}v(t,x))\,dx
\\
&=-\frac{1}{4\pi}\int_{\R^3} u(t,|x|)^{p-1}v(t,x)\,dx
\\
&=-\int_0^\infty u(t,r)^{p}r\,dr.
\end{align*}
The case $N=2$ is delicate because of the lack of boundedness of $\phi$
(even if the equation is linear).
Fix $\eta\in C^\infty(\R)$ satisfying 
$\mathbbm{1}_{(-\infty,0]}\leq \eta \leq \mathbbm{1}_{(-\infty,1]}$
and decreasing in $(0,1)$.
Then we use 
the following compactly supported truncation
$\phi_n(x)=\eta(\frac{\log |x|}{n})\phi(x)$.
By the equation in \eqref{intro:eq1}, we have
\begin{align*}
\frac{d}{dt}\int_{\Omega}u \phi_n\,dx
=
\int_{\Omega}\pa_tu \phi_n\,dx
=
\int_{\Omega}(\Delta u- u^p) \phi_n\,dx
=
\int_{\Omega}u \Delta \phi_n\,dx
-\int_{\Omega}u^p \phi_n\,dx
\end{align*}
which implies
\begin{align*}
&\int_{\Omega}u(t,x) \phi_n(x)\,dx
+
\int_0^t
\int_{\Omega}u(s,x)^p \phi_n(x)\,dx
\,ds
\\
&\quad =
\int_{\Omega}u_0(x)\phi_n(x)\,dx
+
\int_0^t
\int_{\Omega}u(s,x) \Delta \phi_n(x)\,dx
\,ds.
\end{align*}
Noting that the family $\{\phi_n\}_{n=1}^\infty$ 
is a non-decreasing sequence converging to $\phi$ pointwisely and 
satisfies
\begin{align*}
\Delta \phi_n(x)
&=\frac{2}{n}\eta'\left(\frac{\log |x|}{n}\right)\frac{x\cdot\nabla \phi(x)}{|x|^2}
+
\frac{1}{n^2}\eta''\left(\frac{\log |x|}{n}\right)\frac{\phi(x)}{|x|^2}
\end{align*}
which converges uniformly to $0$ in $(0,\infty)\times \Omega$ as $n\to \infty$, 
by the monotone convergence theorem
we obtain the desired equality.  
\end{proof}

The proof for the non-vanishing case is analogous to 
the global existence for the Fujita equation $\pa_tu-\Delta u=u^p$ 
via the comparison principle (see e.g., Quittner--Souplet \cite[Theorem 20.2]{QS}).
\begin{theorem}\label{thm:nonvanishing}
Assume that  $u_0\in  L^1_{\phi}(\Omega) \cap L^\infty(\Omega)$
be nontrivial and nonnegative. 
If $p>\min\{2,1+\frac{2}{N}\}$, 
then the corresponding global-in-time solution of \eqref{intro:eq1}
satisfies 
$\lim\limits_{t\to\infty}M_{\Omega}(u(t))>0$.
\end{theorem}
\begin{proof}
Note that $M_{\Omega}(u_0)>0$. 
We employ the comparison principle with a sub-solution $v(t,x)=h(t)S(t)u_0$
of \eqref{intro:eq1}, where $h$ is a suitable positive $C^1$ function determined as follows. 
Since $v$ satisfies $v|_{\pa\Omega}=0$ and 
\begin{align*}
\pa_tv-\Delta v+v^p
&=
\Big(h'(t)+h(t)^p\big(S(t)u_0\big)^{p-1}\Big)S(t)u_0,
\end{align*}
the function $v$ is sub-solution of \eqref{intro:eq1} if 
$h(0)=1$ and $h'(t)=-h(t)^p\|S(t)u_0\|_{L^\infty}^{p-1}$. 
This consideration suggests that 
\[
h(t)=\left(1+(p-1)\int_0^t\|S(\tau)u_0\|_{L^\infty}^{p-1}d\tau\right)^{-\frac{1}{p-1}}
\]
could be a suitable choice. Therefore the comparison principle shows that $u\geq v$ and then
\begin{align*}
M_\Omega(u(t)) 
&\geq 
M_\Omega(v(t)) 
\\
&= \left(1+(p-1)\int_0^t\|S(\tau)u_0\|_{L^\infty}^{p-1}d\tau\right)^{-\frac{1}{p-1}}M_\Omega(u_0).
\end{align*}
Combining this with Lemmas \ref{lem:linear-Linfty-decay} and \ref{lem:integral_0-t},
we have $\big((1+t)^{\frac{N}{2}}\mathcal{E}_N(t)\big)^{1-p}\in L^1(0,\infty)$ if $p>\min\{2,1+\frac{2}{N}\}$ 
and in this case we can obtain $\lim\limits_{t\to \infty}M_{\Omega}(u(t))>0$. The proof is complete.
\end{proof}

\section{The vanishing case of $M_{\Omega}(u(t))$}
\label{sec:vanishing}

Next we treat the case of the vanishing of $M_{\Omega}(u(\cdot))$.
The idea is to use a modified version of the test function method.
This argument is only applicable for absorbing nonlinearities.

\begin{theorem}\label{thm:vanishing}
Assume that  $u_0\in  L_{\phi}^1(\Omega)\cap L^\infty(\Omega)$ is nonnegative. 
If $1<p\leq \min\left\{2,1+\frac{2}{N}\right\}$, 
then the corresponding  global-in-time solution $u$ of \eqref{intro:eq1}
satisfies 
$\lim\limits_{t\rightarrow\infty}M_{\Omega}(u(t))=0$.
\end{theorem}
Before enter the detail, we describe 
the property of the cut-off functions $\{\varphi_R\}_{R>0}$
defined as 
\begin{equation}\label{eq:varphi_R}
\varphi_R(t,x)=\eta\big(\xi_R(x,t)\big)^{2p'}, 
\quad 
\varphi_R^*(t,x)=\eta^*\big(\xi_R(x,t)\big)^{2p'}, 
\quad
\xi_R(t,x)=\frac{(|x|-R_1)_+^2+t}{R},
\end{equation}
where $\eta\in C^\infty(\R)$ is a non-increasing function 
satisfying 
$\mathbbm{1}_{[1,\infty)} \leq \eta\leq \mathbbm{1}_{[\frac{1}{2},\infty)}$, 
$\eta^*=\mathbbm{1}_{[\frac{1}{2},1]}\eta$, $p'=p/(p-1)$ is the H\"older conjugate of $p$ 
and $R_1=\max\{R_0,R_{\phi}\}>1$ ($N\geq 2$) or $R_1=0$ ($N=1$). 
\begin{lemma}\label{lem:testfunction}
Define the family of cut-off functions $\{\varphi_R\}_{R>0}$ 
as in \eqref{eq:varphi_R}.
Then the following inequality holds:
\[
|\pa_t(\phi(x)\varphi_R(t,x))|
+
|\Delta(\phi(x)\varphi_R(t,x)|\leq \frac{C}{R}
\phi(x)\varphi_R^*(t,x)^{\frac{1}{p}}.
\]
\end{lemma}
\begin{proof}
It is easy to see that $|\phi\pa_t\varphi_R|
\leq \frac{2p'\|\eta'\|_{L^\infty}}{R}\varphi_R^*(t,x)^{\frac{1}{p}}$ for all dimensions.
Therefore we focus our attention to 
\[
\Delta(\phi\varphi_R)
=
(2\nabla \phi\cdot \nabla \xi_R+\phi\Delta\xi_R)\widetilde{\eta}'(\xi_R)
+
\phi|\nabla \xi_R|^2 \widetilde{\eta}''(\xi_R), 
\]
where we have put $\widetilde{\eta}=\eta^{2p'}$ for simplicity. 
Note that 
$\nabla\xi_R=\frac{2x}{R}(1-\frac{R_1}{|x|})_+$
and 
$|\Delta\xi_R|\leq 
\frac{2N}{R}$.
If $N=1$, then one has $\phi(x)=x$ and hence
a direct computation yields 
$|\Delta(\phi\varphi_R)|\leq \frac{C}{R}\phi$, 
where $C$ is a positive constant depending only on $\eta$.
If $N=2$, then Lemma \ref{lem:harmonic} {(ii)} and 
$(1-\frac{R_1}{|x|})\leq \phi_0(\frac{x}{R_1})\leq \phi_0(\frac{x}{R_0})$ on $\R^2\setminus B(0,R_1)$ 
give
\begin{align*}
|\nabla \phi\cdot \nabla \xi_R|\leq \frac{4}{R}\phi, \quad |\nabla \xi_R|^2\leq \frac{4}{R}
\quad \text{on}\ \Big\{(x,t);\xi_R\in [\tfrac{1}{2},1]\Big\}
\end{align*}
which provides the desired inequality. If $N\geq 3$, then 
the computation is similar.
\end{proof}
\begin{proof}[Proof of Theorem \ref{thm:vanishing}]
For simplicity we put $M_{\Omega,\infty}=\lim\limits_{t\to\infty}M_{\Omega}(u(t))$. 
We assume $M_{\Omega,\infty}>0$ and argue by contradiction.
Let $\varphi_R$ be given in Lemma \ref{lem:testfunction}.
Since $\varphi_R(x,t)\rightarrow 0$ pointwisely as $R\rightarrow\infty$, 
by the dominated convergence 
we can choose a constant $R_2\geq R_1$ such that 
for every $R\geq R_2$, 
\begin{gather}\label{eq:neglect-initial}
\int_{\Omega}u_0(x)\phi(x)\varphi_R(0,x)\,dx\leq \frac{1}{2}M_{\Omega,\infty}.
\end{gather}
Then as in the proof of Lemma \ref{lem:M_Omega}, 
we see that for every $T>R$, 
\begin{align*}
&
M_{\Omega}(u(T))-\int_{\Omega}u_0(x)\phi(x)\varphi_R(0,x)\,dx
+\int_0^T\int_{\Omega}u(t,x)^p\phi(x)\varphi_R(t,x)\,dx\,dt
\\
&=
\int_0^T\int_{\Omega}u(t,x)\Big(\phi(x)\partial_t\varphi_R(t,x)+\Delta(\phi(x)\varphi_R(t,x))\Big)\,dx\,dt
\\
&\leq 
\frac{C}{R}
\int_0^T\int_{\Omega}u(t,x)\phi(x)\varphi_R^*(t,x)^{\frac{1}{p}}
\,dx\,dt,
\end{align*}
where we have used $\varphi_R(T,x)=1$.
Letting $T\rightarrow\infty$, 
using the dominated convergence theorem, 
the H\"older inequality and \eqref{eq:neglect-initial},
we arrive at 
\begin{align}
\nonumber
&
\frac{1}{2}M_{\Omega,\infty}
+\int_0^\infty\int_{\Omega}u(t,x)^p\phi(x)\varphi_R(t,x)\,dx\,dt
\\
\nonumber
&\leq 
\frac{C}{R}
\int_0^\infty\int_{\Omega}u(t,x)\phi(x)\varphi_R^*(t,x)^{\frac{1}{p}}
\,dx\,dt
\\
&
\label{eq:testfuntionmethod}
\leq 
\frac{C}{R}
\left(
\int_0^\infty\int_{\Omega}\phi(x)\mathbbm{1}_{[\frac{1}{2},1]}(\xi_R(t,x))
\,dx\,dt
\right)^{1-\frac{1}{p}}
\left(
\int_0^\infty\int_{\Omega}u(t,x)^p\phi(x)\varphi_R^*(t,x)
\,dx\,dt
\right)^{\frac{1}{p}}
\end{align}
for every $R\geq R_2$.
Here we shall introduce the auxiliary function 
\[
Y(R)=
\int_0^R 
\left(
\int_0^\infty\int_{\Omega}u(t,x)^p\phi(x)\phi_\rho^*(t,x)
\,dx\,dt
\right)\frac{d\rho}{\rho}, \quad R>0
\]
which is well-defined and bounded. As in the proof of \cite[Lemma 2.2]{IkedaSobajima2019}, we can see that 
\[Y(R)
\leq \log 2\int_0^\infty\int_{\Omega}u(t,x)^p\phi(x)\varphi_R(t,x)\,dx\,dt, 
\]
and therefore \eqref{eq:testfuntionmethod} yields 
\[
\left(
\frac{1}{2}M_{\Omega,\infty}
+\frac{1}{\log 2}Y(R)\right)^{p}
\leq C^p
\Theta(R)Y'(R), 
\]
where we have put
\[
\Theta(R)=\left(\frac{1}{R}\int_0^R\int_{\Omega\cap\{(|x|-R_1)^2+t\leq R\}}\phi(x)\,dx\,dt\right)^{p-1}.
\]
Solving the above differential inequality, 
we can see that 
to conclude a contradiction, 
it suffices to check that $\Theta^{-1}\notin L^1((R_2,\infty))$ which implies $Y(R)<0$ for
sufficiently large $R$. 
For the case $N\geq 3$, 
we see from $\phi\leq 1$ that 
\begin{align*}
\big(\Theta(R)\big)^{-1}\geq
\left(\frac{1}{R}\iint_{\{(|x|-R_1)^2+t\leq R\}}\,dx\,dt\right)^{1-p}
\geq CR^{-\frac{N}{2}(p-1)}
\end{align*}
which is not integrable if and only if $p\leq 1+\frac{2}{N}$.
For the case $N=1$, we see from the explicit form $\phi(x)=x$ that 
\begin{align*}
\big(\Theta(R)\big)^{-1}
\geq 
\left(\frac{1}{R}\iint_{\{x^2+t\leq R\}}x\,dx\,dt\right)^{1-p}
=CR^{1-p}
\end{align*}
which is not integrable if and only if $p\leq 2$. 
Finally we consider the case $N=2$. In this case, we see that 
for $R>R_1^2$,  
\begin{align*}
\big(\Theta(R)\big)^{-1}
&\geq 
\left(\frac{1}{R}\int_0^R \int_{\Omega \cap\{(|x|-R_1)^2+t\leq R\}}\phi(x)\,dx\,dt\right)^{1-p}
\\
&\geq 
\left(\int_{\{r_0\leq |x|\leq \sqrt{R}+R_1\}}\log\frac{|x|}{r_0}\,dx\right)^{1-p}
\\
&\geq 
\left(2\pi\int_{r_0}^{2\sqrt{R}}r\log\frac{r}{r_0}\,dr\right)^{1-p}
\\
&\geq 
\left(2\pi R\log\frac{4R}{r_0^2}\right)^{1-p}
\end{align*}
which is not integrable if and only if $p\leq 2=1+\frac{2}{N}$. 
The proof is complete.
\end{proof}

\begin{remark}
Theorem \ref{thm:vanishing} can be proved 
via the usual dominated convergence theorem except the case $(N,p)=(2,2)$. 
We emphasize that 
the method with the auxiliary function $Y(R)$ 
enables us to treat the exceptional case $(N,p)=(2,2)$. 
\end{remark}

\section{The asymptotic behavior for non-vanishing case} 
\label{sec:asymptotics}

In the case $p>\min\{2,1+\frac{2}{N}\}$, we define 
\begin{equation}\label{eq:u_infty_def}
u_{\infty}(x)=u_0(x)-\int_0^\infty \big(u(s,x)\big)^p\,ds
\end{equation}
which is well-defined in 
$L^\infty(\Omega)\cap L_{\phi}^1(\Omega)$
when $u_0$ belongs to the same class $L^\infty(\Omega)\cap L_{\phi}^1(\Omega)$.
To find the asymptotic profile of $u$, we further utilize the other following 
auxiliary function $F$. 
\begin{lemma}\label{lem:diffusionstate}
Assume that  $u_0\in  L_{\phi}^1(\Omega)\cap L^\infty(\Omega)$ is nonnegative. 
Let $u$ be the corresponding solution of \eqref{intro:eq1}.
If $p>\min\{2,1+\frac{2}{N}\}$, then the function
\[
F:(t,x)\mapsto \int_t^\infty u(s,x)^p\,ds\in L^\infty(\Omega)\cap L_{\phi}^1(\Omega)
\]
is well-defined. Moreover, 
there exist a positive consntat $C$ such that 
for every $1\leq q\leq \infty$, 
\begin{align*}
&\|F(t)\|_{L^1_\phi}
+
(1+t)\|\pa_tF(t)\|_{L^1_\phi}
\\
&\quad+
(1+t)^{\frac{N}{2}(1-\frac{1}{q})}\mathcal{E}_N(t)
\Big(\|F(t)\|_{L^q}
+
(1+t)\|\pa_tF(t)\|_{L^q}\Big)
\leq 
C|\!|\!| u_0 |\!|\!|^p
\widetilde{\mathcal{E}}_{N}(t), \quad t\geq 0, 
\end{align*}
where $|\!|\!| u_0 |\!|\!|=\|u_0\|_{L^\infty}+\|u_0\|_{L_\phi^1}$
and
\[
\widetilde{\mathcal{E}}_{N}(t)
=
\begin{cases}
(1+t)^{2-p}
&\text{if}\ N=1,
\\
(1+t)^{2-p}\big(1+\log(1+t)\big)^{1-p}
&\text{if}\ N=2,
\\
(1+t)^{1-\frac{N}{2}(p-1)}
&\text{if}\ N\geq 3.
\end{cases}
\]

\end{lemma}
\begin{proof}
By a direct computation with comparison principle, we have 
\begin{align*}
u(s,x)^p
\leq 
\|S(s)u_0\|_{L^\infty}^{p-1}S(s)u_0(x), \quad s>0.
\end{align*}
As is mentioned in the proof of Theorem \ref{thm:nonvanishing},
we see from Lemma \ref{lem:integral_t-infty} that $F\in C^1((0,\infty);L_\phi^1)$ with the estimate 
\begin{align*}
\|F(t)\|_{L^1_\phi}\leq 
C
\left(\|u_0\|_{L^\infty}+\|u_0\|_{L_\phi^1}\right)^p
\int_t^\infty \Big((1+s)^{\frac{N}{2}}\mathcal{E}_{N}(s)\Big)^{1-p}\,ds.
\end{align*}
The $L^q$-estimates ($1\leq q\leq \infty$) are similar.
\end{proof}

\begin{theorem}\label{thm;asymptotics}
Let $N\geq 1$ and $p>\min\{2,1+\frac{2}{N}\}$. 
Then for every $1\leq q\leq \infty$, one has
\[
(1+t)^{\frac{N}{2}(1-\frac{1}{q})}\mathcal{E}_N(t)
\|u(t)-S(t)u_\infty\|_{L^q}\leq 
C\left(
\widetilde{\mathcal{E}}_{N}(t)
+\frac{1}{t}\int_0^t\widetilde{\mathcal{E}}_{N}(s)\,ds\right),
\]
where $u_\infty$ is defined by \eqref{eq:u_infty_def} 
and 
$\widetilde{\mathcal{E}}_{N}(\cdot)$ is given in Lemma \ref{lem:diffusionstate}.
\end{theorem}

\begin{proof}
Thanks to Lemma \ref{lem:diffusionstate},
we can consider the auxiliary problem 
\begin{equation}\label{eq:aux-heat}
\begin{cases}
\pa_t\widetilde{u}-\Delta \widetilde{u}=F&\text{in}\ (0,\infty)\times \Omega, 
\\
\widetilde{u}=0&\text{on}\ (0,\infty)\times \pa\Omega,
\\
\widetilde{u}(0,x)=0 &\text{in}\ \Omega. 
\end{cases}
\end{equation}
Then the corresponding solution $\widetilde{u}$ of the problem \eqref{eq:aux-heat} 
also satisfies
\begin{equation}\label{eq:aux-heat2}
\begin{cases}
\pa_t(\pa_t\widetilde{u})-\Delta (\pa_t\widetilde{u})=\pa_tF(t,x)=-u(t,x)^p &\text{in}\ (0,\infty)\times \Omega, 
\\
\widetilde{u}=0&\text{on}\ (0,\infty)\times \pa\Omega,
\\
\pa_t\widetilde{u}(0,x)=F(0) &\text{in}\ \Omega.
\end{cases}
\end{equation}
This provides that $u(t)-S(t)u_{\infty}=\pa_t \widetilde{u}(t)$, where $u_\infty = u_0-F(0)$.
Therefore 
the problem of the derivation of estimates for $u(t)-S(t)u_{\infty}$ 
is reduced to the one for $\pa_t\widetilde{u}$. 
On the other hand, since by the Duhamel formula, $\widetilde{u}$ can be represented as 
\begin{align*}
\widetilde{u}(t)
&=
\int_0^t S(t-s)F(s)\,ds=\int_0^{\frac{t}{2}} S(t-s)F(s)\,ds+\int_0^\frac{t}{2}S(s)F(t-s)\,ds
\end{align*}
and therefore
\begin{align*}
\pa_t \widetilde{u}(t)
=
\int_\frac{t}{2}^t S(t-s)\pa_tF(s)\,ds
+S(\tfrac{t}{2})F(\tfrac{t}{2})
+\int_0^{\frac{t}{2}} \Delta S(t-s)F(s)\,ds={\bf (I)}+{\bf (II)}+{\bf (III)}.
\end{align*}
We give reasonable estimates for the terms {\bf (I)}, {\bf (II)} and {\bf (III)}, respectively. 
For {\bf (I)} and {\bf (II)}, 
the $L^q$-contraction together with Lemma \ref{lem:diffusionstate}
yields
\begin{align*}
\|{\bf (I)}\|_{L^q}
\leq 
\int_\frac{t}{2}^t \|\pa_tF(s)\|_{L^q}\,ds
\leq 
\int_\frac{t}{2}^\infty \|\pa_tF(s)\|_{L^q}\,ds
\leq 
\frac{C\Big(\|u_0\|_{L^\infty}+\|u_0\|_{L_\phi^1}\Big)^p
\widetilde{\mathcal{E}}_{N}(\frac{t}{2})}{
(1+\frac{t}{2})^{\frac{N}{2}(1-\frac{1}{q})}\mathcal{E}_N(\frac{t}{2})}.
\end{align*}
\[
\|S(\tfrac{t}{2})F(\tfrac{t}{2})\|_{L^q}
\leq
\|F(\tfrac{t}{2})\|_{L^q}
\leq 
\frac{C\Big(\|u_0\|_{L^\infty}+\|u_0\|_{L_\phi^1}\Big)^p
\widetilde{\mathcal{E}}_{N}(\frac{t}{2})}{
(1+\frac{t}{2})^{\frac{N}{2}(1-\frac{1}{q})}\mathcal{E}_N(\frac{t}{2})},
\]
respectively.
For {\bf (III)}, 
by Lemmas \ref{lem:diffusionstate} and \ref{lem:linear-Linfty-decay},
we see from 
the property of analytic positive semigroup that
\begin{align*}
\|{\bf (III)}\|_{L^q}
&\leq 
C\int_0^{\frac{t}{2}} 
\|\Delta S(\tfrac{t-s}{2})\|_{L^q\to L^q}
\|S(\tfrac{t-s}{2})F(s)\|_{L^q}\,ds
\\
&\leq 
\frac{C}{(1+\frac{t}{4})^{\frac{N}{2}(1-\frac{1}{q})}
\mathcal{E}_N(\frac{t}{4})}
\frac{1}{t}\int_0^\frac{t}{2}\Big(\|F(s)\|_{L_\phi^1}+\|F(s)\|_{L^q}\Big)\,ds,
\\
&\leq 
\frac{C}{(1+t)^{\frac{N}{2}(1-\frac{1}{q})}
\mathcal{E}_N(t)}
\frac{1}{t}\int_0^\frac{t}{2}\widetilde{\mathcal{E}}_{N}(s)\,ds.
\end{align*}
Summarizing the above estimates, we obtain 
\[
(1+t)^{\frac{N}{2}(1-\frac{1}{q})}\mathcal{E}_N(t)
\|u(t)-S(t)u_\infty\|_{L^q}\leq 
C\left(
\widetilde{\mathcal{E}}_{N}(t)
+\frac{1}{t}\int_0^t\widetilde{\mathcal{E}}_{N}(s)\,ds\right).
\]
The proof is complete.
\end{proof}

\begin{remark}
If $p>\min\{3,1+\frac{4}{N}\}$, then a new auxiliary function 
\[
F_*:(t,x)\to -\int_t^\infty F(s,x)\,ds
\]
is also well-defined and $\pa_t^2F_*=-u^p$. This gives that 
the solution $\widetilde{u}_*$ of the other auxiliary problem
\begin{equation}\label{eq:aux-heat-extra}
\begin{cases}
\pa_t\widetilde{u}_*-\Delta \widetilde{u}_*=F_*&\text{in}\ (0,\infty)\times \Omega, 
\\
\widetilde{u}_*=0&\text{on}\ (0,\infty)\times \pa\Omega,
\\
\widetilde{u}_*(0,x)=0 &\text{in}\ \Omega
\end{cases}
\end{equation}
satisfies
\begin{equation}\label{eq:aux-heat-extra-2}
\begin{cases}
\pa_t(\pa_t^2\widetilde{u}_*)-\Delta (\pa_t^2\widetilde{u}_*)=-u^p&\text{in}\ (0,\infty)\times \Omega, 
\\
\pa_t^2\widetilde{u}_*=0&\text{on}\ (0,\infty)\times \pa\Omega,
\\
\pa_t^2\widetilde{u}_*(0,x)=F(0)+\Delta F_*(0) &\text{in}\ \Omega. 
\end{cases}
\end{equation}
This consideration suggests that 
by estimating 
$\pa_t^2\widetilde{u}_*(t)=u(t)-S(t)u_\infty+\Delta S(t)F_*(0)$,
we could have a refined decay rate faster than $S(t)u_\infty$,
which justifies the asymptotic profile of second-order. 
Alhough a higher-order expansion also could be considered in a similar procedure 
for the case $p>\min\{1+k,1+\frac{2k}{N}\}$ ($k\in\N$), 
we do not enter the detail
because this attempt seems independent of the (non-)vanishing of the mass with respect to 
the invariant measure $\phi(x)\,dx$.
\end{remark}

Finally, we prove Corollary \ref{cor:Ngeq3}.
\begin{proof}[Proof of Corollary \ref{cor:Ngeq3}]
 If $N\geq 3$, then we can additionally find the 
explicit asymptotic behavior 
by just applying \cite[Theorem 1.1]{DR2025JDE} and using Lemma \ref{lem:unique-harmonic}:
for every $f\in L^1(\Omega)$ 
and $1\leq q\leq \infty$, 
\begin{equation}\label{eq:DR2025JDE}
\lim_{t\to \infty}\Big((1+t)^{\frac{N}{2}(1-\frac{1}{q})}\|S(t)f-M_\Omega(f)\phi(\cdot) G(t,\cdot)\|_{L^q}\Big)=0,
\end{equation}
where we recall $M_{\Omega}(f)=\int_{\Omega}f\phi\,dx$. Therefore 
combining Theorem \ref{thm;asymptotics} and \eqref{eq:DR2025JDE} with $f=u_\infty$, 
we can obtain the desired assertion.
\end{proof}


\subsection*{Declarations}
\begin{description}
\item[Data availability] Data sharing not applicable to this article as no datasets were generated or analysed during the current study.
\item[Conflict of interest] 
The authors declare that they have no conflict of interest.
\end{description}

{\small 

\end{document}